\documentclass[a4paper, 11 pt]{article}

\usepackage{amsfonts,amsmath,amssymb,hhline,ifthen,amsthm,graphics,cite,latexsym, caption2}
\usepackage{latexsym}
\usepackage[T2A]{fontenc}
\usepackage[cp1251]{inputenc}
\usepackage[english]{babel}
\usepackage[matrix, arrow, curve]{xy}
\usepackage{graphicx}

\newcounter{q3}
\newcommand{\dsm}[3]{
{\if#20{\if#31{\frac{\partial #1}{\partial y}}\else
          {\frac{\partial^{#3} #1}{\partial y^{#3}}}
        \fi}\else
  {\if#30{\if#21{\frac{\partial #1}{\partial x}}\else
            {\frac{\partial^{#2} #1}{\partial x^{#2}}}
          \fi}\else
    {\setcounter{q3}{#2}\addtocounter{q3}{#3}
    \frac{\partial{\if{1}\arabic{q3}^{}\else^{ \arabic{q3} }\fi}#1}
    {{\if#20\else{\partial x{\if#21\else{^{#2}}\fi}}\fi}
     {\if#30\else{\partial y\if#31\else{^{#3}}\fi}\fi} }}
   \fi}
\fi} }

\newcommand{\dzm}[2]{
{\if#10{\if#21{\frac{\partial}{\partial\overline{\zeta}}}\else
          {\frac{\partial^{#2}}{\partial\overline{\zeta}^{#2}}}
        \fi}\else
  {\if#20{\if#11{\frac{\partial}{\partial\zeta}}\else
            {\frac{\partial^{#1}}{\partial\zeta^{#1}}}
          \fi}\else
    {\setcounter{q3}{#1}\addtocounter{q3}{#2}
    \frac{\partial{\if{1}\arabic{q3}^{}\else^{ \arabic{q3} }\fi}}
    {{\if#10\else{\partial\zeta{\if#21\else{^{#1}}\fi}}\fi}
     {\if#20\else{\partial\overline{\zeta}\if#21\else{^{#2}}\fi}\fi} }}
   \fi}
\fi} }

\newcounter{q2}

\newtheoremstyle{theor}% name
  {\medskipamount}%      Space above
  {\medskipamount}%      Space below
  {\itshape}% Body font
  {}%\parindent}% Indent amount (empty = no indent, \parindent = para indent)
  {\bfseries}
  {.}
  {.5em}
  {}

\newtheorem{definition}{Definition}[section]
\newtheorem{theorem}[definition]{Theorem}
\newtheorem{lemma}[definition]{Lemma}
\newtheorem{proposition}[definition]{Proposition}
\newtheorem{corollary}[definition]{Corollary}

\theoremstyle{definition}
\newtheorem{remark}[definition]{Remark}
\newtheorem{example}[definition]{Example}
\numberwithin{equation}{section}

\makeindex

\newtheoremstyle{remarks}% name
  {0mm}%      Space above
  {0mm}%      Space below
  {\itshape}% Body font
  {}%\parindent}% Indent amount (empty = no indent, \parindent = para indent)
  {\itshape}
  {.}
  {.5em}
  {}
\makeatletter \@addtoreset{equation}{section} \makeatother

\begin{document}

\subsection*{\center BOUNDEDNESS OF LEBESGUE CONSTANTS AND INTERPOLATING FABER BASES}\begin{center}\textbf{V. Bilet, O. Dovgoshey and J. Prestin} \end{center}
\parshape=5
1cm 10.5cm 1cm 10.5cm 1cm 10.5cm 1cm 10.5cm 1cm 10.5cm \noindent \small {\textbf{Abstract.}  We investigate some conditions under which the Lebesgue constants or Lebesgue functions are bounded for the classical Lagrange polynomial interpolation on a compact subset of $\mathbb R$. In particular, relationships of such boundedness with uniform and pointwise convergence of Lagrange polynomials and with the existence of interpolating Faber bases are discussed.
\bigskip

\parshape=5
1cm 10.5cm 1cm 10.5cm 1cm 10.5cm 1cm 10.5cm 1cm 10.5cm \noindent \small {\bf Key words:} Lebesgue constant, Lebesgue function, Lagrange polynomial interpolation, interpolating Faber basis.

\parshape=2
1cm 10.5cm 1cm 10.5cm  \noindent \small {\bf }

% \bigskip
%\textbf{ AMS 2010 Subject Classification: }

\large\section{Introduction}
\hspace*{\parindent}
Let
\begin{equation*}
\mathfrak M=\{x_{k,n}\}=\begin{pmatrix}
    x_{1,1}\\
    x_{1,2}&x_{2,2}\\
    ...&...&...\\
    x_{1,n}&x_{2,n}&...&x_{n,n}\\
    ...&...&...&...&...\\
    \end{pmatrix}
\end{equation*}
be an infinite triangular matrix whose elements (nodes) are real numbers satisfying the condition $x_{k_1, n}\ne x_{k_2, n}$ for all distinct $k_1, k_2\in\{1, ..., n\}$ and every $n\in\mathbb N.$
Then define the fundamental polynomials $l_{k_0, n}=l_{k_0, n}(\mathfrak M, \cdot)$ as
\begin{equation}\label{eq1.1}
l_{k_0, n}(x)=l_{k_0, n}(\mathfrak M, x):=
\prod_{1\le k\le n, \, k\ne k_0}\frac{(x-x_{k, n})}{(x_{k_0, n}-x_{k, n})}, \quad x\in\mathbb R.
\end{equation}
The polynomials $l_{1, n}, ..., l_{n, n}$ form a basis at the linear space $H_{n-1}$ of all real algebraic polynomials of degree at most $n-1.$  %In what follows we shall say that $l_{k, n}(\mathfrak M, \cdot)$ are the fundamental polynomials corresponding to $\mathfrak M.$
In particular, we have $l_{1, 1}\equiv 1.$

Let $X$ be an infinite compact subset of $\mathbb R.$ Denote by $C_{X}$ the Banach space of continuous functions $f: X\to\mathbb R$ with the supremum norm
\begin{equation*}
\parallel f\parallel_{X}:=\sup\{|f(x)|: x\in X\}
\end{equation*} and write $\mathfrak M\subseteq X$ if $\mathfrak M=\{x_{k,n}\}$ and $x_{k,n}\in X$ for all $n\in\mathbb N$ and $k\le n.$ For $f\in C_X,$ $\mathfrak M\subseteq X$ and $n\in\mathbb N,$ the Lagrange interpolating polynomial $L_{n}(f, \mathfrak M, \cdot)$ is the unique polynomial from $H_n$ which coincedes with $f$ at the nodes $x_{k, n+1}, k=1, ..., n+1.$ Using the fundamental polynomials we can represent $L_{n}(f, \mathfrak M, \cdot)$ in the form
\begin{equation}\label{eq1.2}
L_{n}(f, \mathfrak M, \cdot)=\sum_{k=1}^{n+1}f(x_{k, n+1})l_{k, n+1}(\mathfrak M, \cdot).
\end{equation}

For given $X,$ $\mathfrak M\subseteq X$, and $n\in\mathbb N,$ the Lebesgue function $\lambda_{n}(\mathfrak M, \cdot)$ and the Lebesgue constant $\Lambda_{n, X}(\mathfrak M)$ can be defined as
\begin{equation}\label{eq1.3}
\lambda_{n}(\mathfrak M, x):=\sup\{|L_{n}(f, \mathfrak M, x)|: \,\, \parallel f\parallel_{X}\le 1\},\, x\in\mathbb R,
\end{equation}
and, respectively, as
\begin{equation}\label{eq1.3*}
\Lambda_{n, X}(\mathfrak M):=\sup\{\lambda_{n}(\mathfrak M, x): x\in X\}.
\end{equation}
The mappings \begin{equation}\label{eq1.4*}\mathfrak{L}_{n, \mathfrak M}: C_X \to C_X \quad\mbox{with}\quad \mathfrak{L}_{n, \mathfrak M}(f)=L_{n}(f, \mathfrak M, \cdot)\end{equation} are bounded linear operators having the norms \begin{equation}\label{neweq}\parallel \mathfrak{L}_{n, \mathfrak M}\parallel=\Lambda_{n, X}(\mathfrak M).\end{equation}
For every infinite compact set $X\subseteq\mathbb R$ and $\mathfrak M\subseteq X$ it is easy to prove that the equality
\begin{equation}\label{eq1.4}
\lambda_{n}(\mathfrak M, x)=\sum_{k=1}^{n+1}|l_{k, n+1}(\mathfrak M, x)|
\end{equation}
holds for each $x\in \mathbb R.$

\begin{remark}
Using formulas \eqref{eq1.1}, \eqref{eq1.3*} and \eqref{eq1.4}, we can define the Lebesgue functions $\lambda_{n}(\mathfrak M, \cdot)$ and the Lebesgue constants $\Lambda_{n, X}(\mathfrak M)$ for arbitrary nonempty set $X\subseteq\mathbb R$ and any interpolation matrix $\mathfrak M\subseteq \mathbb R.$
\end{remark}

In what follows we will denote by $\textbf{\emph{BLC}}$ (bounded Lebesgue constants) the set of compact nonvoid sets $X\subseteq [-1, 1],$ for each of which there is a matrix $\mathfrak M\subseteq [-1, 1],$ such that the corresponding sequence $\left(\Lambda_{n, X}(\mathfrak M)\right)_{n\in\mathbb N}$ is bounded, i.e.,
\begin{equation}
\Lambda_{n, X}(\mathfrak M)< c
\end{equation}
holds for some $c>0$ and every $n\in\mathbb N.$

In the second section of the paper we will describe some details of the well-known interplay between the boundedness of Lebesgue constants $\Lambda_{n, X}(\mathfrak M)$
and the uniform convergence of Lagrange polynomials $L_{n}(f, \mathfrak M, \cdot).$ The corresponding relationships of pointwise boundedness of Lebesgue functions $\lambda_{n}(\mathfrak M, \cdot)$ with pointwise convergence of these polynomials are also described. Moreover, the second section contains a discussion of the known results describing the smallness of sets belonging to $\textbf{\emph{BLC}}$.

In the third section we obtain some new relations between the boundedness of $\Lambda_{n, X}(\mathfrak M)$ for special interpolating matrices $\mathfrak M$ and the existence of interpolating Faber bases in the space $C_X.$

\section{Boundedness and convergence in Lagrange interpolation}
\hspace*{\parindent} J. Szabados and P.~V\'{e}rtesi, \cite{SV}, write: ``... in the convergence behavior of the Lagrange interpolatory polynomials ... the Lebesgue functions ... and the Lebesgue constants ... are of fundamental importance...''.

\begin{proposition}\label{Prop1.1}
Let $X$ be an infinite compact subset of $\mathbb R$ and let $\mathfrak M\subseteq X.$ The following statements are equivalent.
\item[\rm(i)]\textit{The inequality \begin{equation*} \limsup_{n\in\mathbb N}\Lambda_{n, X}(\mathfrak M)<\infty\end{equation*} holds.}
\item[\rm(ii)]\textit{The limit relation \begin{equation} \label{1eqPr}\lim_{n\to\infty}\parallel f - L_{n}(f, \mathfrak M,\cdot) \parallel_{X}=0\end{equation} is valid for every $f\in C_{X}.$}
\item[\rm(iii)]\textit{The inequality
\begin{equation}\label{1eqPr*}
\limsup_{n\to\infty}\parallel L_{n}(f, \mathfrak M, \cdot)\parallel_{X}<\infty
\end{equation}
holds for every $f\in C_X.$}
\end{proposition}
\begin{proof}

The linear operator $\mathfrak{L}_{n, \mathfrak M}$ is a projection of $C_{X}$ onto $H_{n}.$ Hence, by Lebesgue's lemma, see \cite[Ch.~2, Pr.~4.1]{DeVL}, we have the inequality
\begin{equation}\label{2eqPr}
\parallel L_{n}(f, \mathfrak M,\cdot) - f\parallel_{X}\le(1+\Lambda_{n, X}(\mathfrak M)) E_{n}(f)
\end{equation}
where $E_{n}(f)$ is the error of the best approximation of $f$ by $H_{n}$ in $C_{X}.$ By the Stone-Weierstrass theorem, the continuous function $f$ is uniformly approximable by polynomials on $X$, i.e.,
$
\mathop{\lim}\limits_{n\to\infty}E_{n}(f)=0.
$
Now $(\textrm{i})\Rightarrow (\textrm{ii})$ follows.

The implication (ii)$\Rightarrow$ (iii) is trivial.

Suppose that (iii) holds.To prove $(\textrm{iii})\Rightarrow (\textrm{i})$ note that equality \eqref{1eqPr*} implies the boundedness of sequences $$\left(\parallel \mathfrak L_{n, \mathfrak M}(f)\parallel_{X}\right)_{n\in\mathbb N}=\left(\parallel L_{n}(f, \mathfrak M, \cdot)\parallel_{X}\right)_{n\in\mathbb N}$$ for every $f\in C_{X}.$ Since all $\mathfrak L_{n, \mathfrak M}: C_X \to C_X$ are continuous linear operators and $C_X$ is a Banach space, the Banach-Steinhaus theorem gives us the inequality
\begin{equation*}
\sup_{n\in\mathbb N}\parallel \mathfrak{L}_{n, \mathfrak M}\parallel <\infty.
\end{equation*}
The last inequality and \eqref{neweq} imply (i).
\end{proof}

There is a pointwise analog of Proposition~\ref{Prop1.1}

\begin{proposition}\label{Prop 1.1*}
Let $X$ be an infinite compact subset of $\mathbb R$ and let $x\in X.$ The following statements are equivalent for every $\mathfrak M\subseteq X$.
\newline \emph{(i)} The inequality
\begin{equation}\label{eq1.11*}
\limsup_{n\to\infty}\lambda_{n}(\mathfrak M, x)<\infty
\end{equation}
holds.
\newline \emph{(ii)} The limit relation
\begin{equation*}
\lim_{n\to\infty}L_{n}(f, \mathfrak M, x)=f(x)
\end{equation*}
is valid for every $f\in C_{X}$.
\newline \emph{(iii)} The inequality
\begin{equation}
\limsup_{n\to\infty}|L_{n}(f, \mathfrak M, x)|<\infty
\end{equation}
holds for every $f\in C_{X}$.
\end{proposition}

\begin{proof}
Using \eqref{eq1.3} instead of \eqref{neweq} and the inequality
\begin{equation*}
|f(x)-L_{n}(f, \mathfrak M, x)|\le (1+\lambda_{n}(\mathfrak M, x))E_{n}(f)
\end{equation*}
(see \cite[p.~6]{SV}) instead of \eqref{2eqPr}, we can prove (i) $\Rightarrow$ (ii) as in the proof of Proposition~\ref{Prop1.1}. The implication (ii) $\Rightarrow$ (iii) is trivial. The Banach-Steinhaus theorem and \eqref{eq1.3} give us the implication (iii) $\Rightarrow$ (i).
\end{proof}

\begin{corollary}
Let $X$ be an infinite compact subset of $\mathbb R$ and let $\mathfrak M\subseteq X.$ The sequence $\left(\lambda_{n}(\mathfrak M, \cdot)\right)_{n\in\mathbb N}$ is pointwise bounded on $X$ if and only if the sequence $\left(L_{n}(f,\mathfrak M, \cdot)\right)_{n\in\mathbb N}$ is pointwise convergent to $f$ on $X$ for every $f\in C_X.$
\end{corollary}

For the classical case $X=[-1, 1]$ there exists a lot of important results connected with the unboundedness of the Lebesgue constants and the Lebesgue functions.

In 1914 G.~Faber \cite{Faber}, for every matrix $\mathfrak M\subseteq [-1, 1],$ proved the existence of $f\in C_{[-1,1]}$ satisfying the inequality
\begin{equation}
\limsup_{n\to\infty}\parallel f- L_{n}(f, \mathfrak M, \cdot)\parallel_{[-1, 1]}>0
\end{equation}
that, by Proposition~\ref{Prop1.1}, is an equivalent for
\begin{equation}
\limsup_{n\to\infty}\Lambda_{n, [-1, 1]}(\mathfrak M)=\infty.
\end{equation}
At 1931, S. N. Bernstein \cite{Be1} found that for every $\mathfrak M\subseteq [-1, 1]$ there are $f\in C_{[-1, 1]}$ and $x_0\in [-1, 1]$ such that
\begin{equation}\label{eq1.9}
\limsup_{n\to\infty}|L_{n}(f, \mathfrak M, x_0)|=\infty.
\end{equation}
This equality together with Proposition~\ref{Prop 1.1*} gives the existence of a point $x_0\in [-1, 1]$ satisfying
\begin{equation}\label{eq1.10}
\limsup_{n\to\infty}\lambda_{n}(\mathfrak M, x_0)=\infty.
\end{equation}
In 1980 P.~Erd\"{o}s and P.~V\'{e}rtesi \cite{EV} proved the following
\begin{theorem}\label{Th1.1}
Let $\mathfrak M\subseteq [-1, 1].$ Then there is $f\in C_{[-1, 1]}$ such that limit relation \eqref{eq1.9} holds for almost all $x_0\in [-1, 1].$
\end{theorem}

This theorem implies the following corollary.

\begin{corollary}\label{C1}
Let $X$ be an infinite compact subset of $\mathbb R.$ Let us denote by $m_{1}(X)$ the one-dimensional Lebesgue measure of $X.$ Write $$a=\min\{x: x\in X\} \,\, \mbox{and} \,\, b=\max\{x: x\in X\}.$$ If there is $\mathfrak M\subseteq [a, b]$ such that inequality \eqref{eq1.11*}
holds for every $x\in X,$ then $X$ is nowhere dense and
\begin{equation}\label{eq1.11}
m_{1}(X)=0.
\end{equation}

\end{corollary}

\begin{proof}
Since the fundamental polynomials are invariant under the affine trasformations of $\mathbb R$, we may suppose that $a=-1$ and $b=+1.$
Now, \eqref{eq1.11} follows from Theorem~\ref{Th1.1}. Equality \eqref{eq1.11} implies that the interior of $X$ is empty, $Int X=\varnothing.$ Since $X$ is compact, we have $\overline{X}=X,$ where $\overline{X}$ is the closure of $X$. Consequently, the equality $Int\overline{X}=\varnothing$ holds, it means that $X$ is nowhere dense.
\end{proof}

\begin{corollary}\label{C3}
If $X$ belongs to $\textbf{BLC}$, then $X$ is nowhere dense in $\mathbb R$ and its one-dimensional Lebesgue measure is zero.
\end{corollary}

\begin{example}If $X=\{x_1, x_2, ..., x_k, x_{k+1}, ...\}$ is a dense subset of $[-1, 1]$
and the matrix $\mathfrak M$ is defined such that $x_{k, n}=x_k$ for all $n\in\mathbb N$ and $k\in\{1, ..., n\},$ then we evidently have the equalities
\begin{equation}\label{eq1.12}
\limsup_{n\to\infty}\lambda_{n}(\mathfrak M, x)=\lim_{n\to\infty}\lambda_{n}(\mathfrak M, x)=1
\end{equation}
for every $x\in X.$ Consequently, the compactness of $X$ cannot be dropped in Corollary~\ref{C1}.\end{example}

It was proved by A.~A.~Privalov in \cite{Pr64}, that there are a countable set $X\subseteq [0, 1]$ and a positive constant $c_1=c_{1}(X),$ such that $0$ is the unique accumulation point of $X$ and the inequality
\begin{equation*}
\Lambda_{n, X}(\mathfrak M)\ge c_{1}\ln (n+1)
\end{equation*}
holds for every $n\in\mathbb N$ and every $\mathfrak M\subseteq[-1, 1].$

\begin{remark}
There is a constant $c_{2}>0$ for which $$\Lambda_{n, [-1,1]}(\mathfrak M)\le c_{2}\ln (n+1)$$ holds for every $n\in\mathbb N$ with $\mathfrak M=\{x_{k,n}\}$ based on the Chebyshev nodes $x_{k, n}=\cos\frac{(2k-1)\pi}{2n}.$ For details see \cite{Br}.
\end{remark}

An example of perfect set $X\in \textbf{\emph{BLC}}$ was obtained by S.~N.~Mergelyan \cite{Mr}.

P.~P.~Korovkin \cite{Kr} found a perfect $X\subseteq[-1, 1]$ and a matrix $\mathfrak M$ such that, for every $f\in C_X,$ the sequence $\left(L_{n^2}(f, \mathfrak M, \cdot)\right)_{n\in\mathbb N}$ uniformly tends to $f$,
\begin{equation*}
\sup_{n\in\mathbb N}\Lambda_{n^{2}, X}(\mathfrak M)<\infty.
\end{equation*}
At the same paper \cite{Kr}, he wrote that there is a modification of $X$ with bounded sequence of Lebesgue constants.

Corollary~\ref{C3} indicates that every $X\in\textbf{\emph{BLC}}$ must be small in a very strong sense. Moreover, the examples of A.~A.~Privalov, P.~P. Korovkin and S.~N. Mergelyan show that the properties ``be countable'' and ``belong to the class $\textbf{\emph{BLC}}$'' not linked too closely.

In the rest of the present section we discuss the desirable smallness of sets in terms of porosity.

Let us recall the definition of the right lower porosity at a point.

\begin{definition} Let $X$ be a subset of $\mathbb R$ and let $x_0\in X.$ The right lower porosity of $X$ at $x_0$ is the number
\begin{equation*}
\underline{p}^{+}(X, x_0):=\liminf_{r\to 0^{+}}\frac{\lambda(X, x_0, r)}{r}
\end{equation*}
where $\lambda(X,x_0, r)$ is the length of the largest open subinterval of the set $$(x_0, x_{0}+r)\setminus X=\{x\in (x_0, x_{0}+r): x\notin X\}.$$
\end{definition}
Replacing $(x_0, x_{0}+r)$ in the above definition by the interval $(x_{0}-r, x_0)$, we encounter the notion of the the \emph{left lower porosity} $\underline{p}^{-}(X, x_0).$
The \emph{lower porosity} of $X$ at $x_0$ is the number
\begin{equation*}
\underline{p}(X, x_0):=\max\{\underline{p}^{+}(X, x_0), \underline{p}^{-}(X, x_0)\}.
\end{equation*}
The set $X$ is \emph{strongly lower porous} if $\underline{p}(X, x_0)=1$ holds for every $x_0\in X.$

Let us consider now a modification of the lower porosity.
Write
\begin{equation}\label{eq2.8}
\underline{p}^{*}(X, x_0):=\min\{\underline{p}^{+}(X, x_0), \underline{p}^{-}(X, x_0)\}.
\end{equation}
\begin{theorem}
Let $X$ be a compact subset of $[-1, 1].$ If the inequality
\begin{equation}\label{eq2.9}
\underline{p}^{*}(X, x_0)> \frac{1}{2}
\end{equation}
holds for every $x_0\in X,$ then $X\in \textbf{BLC}.$
%we can find $\mathfrak M\subseteq\mathbb R$ such that the sequence $(\Lambda_{n, X}(\mathfrak M))_{n\in\mathbb N}$ is bounded.
\end{theorem}

\begin{proof}
It is known that
\begin{equation*}
\underline{p}^{+}(X, x_0)>\frac{1}{2}
\end{equation*}
holds if and only if there is $\varepsilon >0,$ which satisfies the condition
\begin{equation*}
X\cap (x_0, x_{0}+\varepsilon]=\varnothing.
\end{equation*}
(See, for example, \cite[Corollary~5.5]{ADK}). Similarly, from $\underline{p}^{-}(X, x_0)>\frac{1}{2}$ it follows that $[x_{0}-\varepsilon, x_0)\cap X= \varnothing$ for some $\varepsilon >0.$ Hence, \eqref{eq2.9} implies that all points of $X$ are isolated. Thus $X$ is discrete. Every compact discrete set is finite. Let $\{x_1, ..., x_n, x_{n+1}, ...\}\subseteq [-1, 1]$ be a countable compact superset of $X$ and let $\mathfrak M=\{x_{k, n}\}$ with $x_{k, n}=x_n$ for all $n\in\mathbb N$ and $k\in\{1, ..., n\}.$ Then there is $n_0\in\mathbb N$ such that $$\lambda_{n, X}(\mathfrak M, x)=1$$ for all $n\ge n_0$ and $x\in X.$ The boundedness of $(\Lambda_{n, X}(\mathfrak{M}))_{n\in \mathbb N}$ follows. Thus, $X$ belongs to $\textbf{\emph{BLC}}.$
\end{proof}

\begin{theorem}\label{Pr2.2}
There is an infinite strongly lower porous compact set $X\subseteq[-1, 1]$ such that $X\not\in \textbf{BLC}.$
%\begin{equation*}
%\limsup_{n\to\infty}\Lambda_{n, X} (\mathfrak M)=\infty
%\end{equation*}
%holds for every interpolation matrix $\mathfrak M\subseteq [-1, 1]$.
\end{theorem}
\begin{proof}
Let $X$ be the compact set, constructed by A.~A.~Privalov in \cite{Pr64}. Then $X\subseteq [0, 1]$ and $0$ is the unique accumulation point of $X$. Note that $\underline{p}^{-}(X, x_0)=1$ holds if and only if $x_0$ is an isolated point of the set $(-\infty, x_0]\cap X.$ Hence, for every $x_0\in X$ we evidently have $\underline{p}^{-}(X, x_0)=1.$ Thus $X$ is strongly lower porous by the definition.
\end{proof}

\section{Faber bases and Lagrange polynomials}

\hspace*{\parindent} In what follows we study the boundedness of the Lebesgue constants $\Lambda_{n, X}(\mathfrak M)$ for the matrices $\mathfrak M$ having the form
\begin{equation*}
\begin{pmatrix}
    x_{1}\\
    x_{1}&x_{2}\\
    ...&...&...\\
    x_{1}&x_{2}&...&x_{n}\\
    ...&...&...&...&...\\
    \end{pmatrix}.
\end{equation*}
The obtained results are inspired by some ideas of J.~Obermaier and R.~Szwarc \cite{Ob}, \cite{OS}.

Let $X$ be an infinite compact subset of $\mathbb R.$

\begin{definition}\label{r2d2.1}
A Faber basis in $C_X$ is a sequence $\tilde p=(p_k)_{k\in\mathbb N}$ of real algebraic polynomials satisfying the following conditions:
\newline \emph{(i)} For every $f\in C_X$ there is a unique sequence $(a_k)_{k\in\mathbb N}$ of real numbers such that
\begin{equation}\label{e1r2}
f=\sum_{k=1}^{\infty}a_k p_k;
\end{equation}
\newline \emph{(ii)} For every $k\in\mathbb N$ the polynomial $p_k$ has the degree $k-1,$ $\emph{deg} p_k=k-1.$
\end{definition}

\begin{remark}
As usual, equality \eqref{e1r2} means that $$\mathop{\lim}\limits_{n\to\infty}\parallel f -\mathop{\sum}\limits_{{k=1}}\limits^{n}a_k p_k \parallel_{X}=0.$$
\end{remark}

Let $\tilde p=(p_k)_{k\in\mathbb N}$ be a Faber basis in $C_X.$
For every $f\in C_X$ we shall denote by $S_{n, \tilde p}(f)$ the partial sum $\mathop{\sum}\limits_{{k=1}}\limits^{n}a_k p_k$ of series \eqref{e1r2}, i.e.,
$$S_{n, \tilde p}(f)=\sum_{k=1}^{n}a_{k}p_{k}.$$ If $n\in\mathbb N$ is given, then the partial sum operator $S_{n, \tilde p}: C_X\to C_X$ is a linear operator with the range $H_{n-1}$ and the domain $C_X$. Similarly, for an interpolation matrix $\mathfrak M\subseteq X,$ the operator, defined by \eqref{eq1.4*},
\begin{equation*}
\mathfrak L_{n, \mathfrak M}: C_X\to C_X,
\end{equation*}
has the same range and domain. Moreover, the linear operators $S_{n, \tilde p}$ and $\mathfrak L_{n, \mathfrak M}$ are projections on $H_{n-1}$, i.e., we have
\begin{equation*}
S_{n, \tilde p}(p)=\mathfrak L_{n, \mathfrak M}(p)=p
\end{equation*}
for every $p\in H_{n-1}.$ In what follows we study some conditions under which the operators $S_{n, \tilde p}$ and $\mathfrak L_{n, \mathfrak M}$ are the same for every $n\in\mathbb N.$

\begin{definition}
A Faber basis $\tilde p=(p_k)_{k\in\mathbb N}$ is interpolating if there is a sequence $(x_k)_{k\in\mathbb N}$ of distinct points of $X$ such that the equality
\begin{equation}\label{e2*.r2}
S_{k, \tilde p}(f)(x_k)=f(x_k)
\end{equation}
holds for all $f\in C_X$ and $k\in\mathbb N.$
\end{definition}
If $\tilde p$ and $(x_k)_{k\in\mathbb N}$ satisfy the above condition, then we say that $\tilde p$ is interpolating with the nodes $(x_k)_{k\in\mathbb N}.$
\begin{remark}\label{rem1.r2}
The interpolating Faber bases are a particular case of the interpolating Schauder bases for a space of continuous functions on a locally compact metric space, \cite[Definition~1.3.1]{SZ}.
\end{remark}

The following lemma is similar to Proposition~1.3.2 from \cite{SZ}.

\begin{lemma}\label{lem1.r2}
Let $X$ be an infinite compact subset of $\mathbb R$, let $\tilde p=(p_k)_{k\in\mathbb N}$ be a Faber basis in $C_X$ and let $(x_k)_{k\in\mathbb N}$ be a sequence of distinct points of $X.$ Then $\tilde p$ is interpolating with the nodes $(x_k)_{k\in\mathbb N}$ if and only if
\begin{equation}\label{e3*.r2}
p_{k}(x_k)\ne 0 \quad \mbox{and}\quad p_{k}(x_j)=0
\end{equation}
for every $k\in\mathbb N$ and $j<k.$
\end{lemma}
\begin{proof}
Suppose that $\tilde p$ is interpolating with the nodes $(x_k)_{k\in\mathbb N}.$ We must show that \eqref{e3*.r2} holds for all $k\in\mathbb N$ and $j<k.$ Since, for each $f\in C_X$, the representation
\begin{equation}\label{e4*.r2}
f=\sum_{k=1}^{\infty}a_k p_k
\end{equation}
is unique, we have
\begin{equation}\label{e5*.r2}
p_k\ne 0
\end{equation}
for every $k\in\mathbb N.$ The equality deg$p_1=0$ together with \eqref{e5*.r2} implies \eqref{e3*.r2} for $k=1.$ Let $k\ge 2.$ The uniqueness of representation \eqref{e4*.r2} gives us the equalities
\begin{equation}\label{e6*.r2}
S_{1, \tilde p}(p_k)=... =S_{k-1, \tilde p}(p_k)=0.
\end{equation}
Since $\tilde p$ is interpolating with the nodes $(x_k)_{k\in\mathbb N},$ \eqref{e6*.r2} implies
$$p_{k}(x_1)=...=p_{k}(x_{k-1})=0.$$
If $p_{k}(x_k)=0,$ then $p_k$ has $k$ distinct zeros that contradicts the equality deg$p_k= k-1$. Condition \eqref{e3*.r2} follows.

Let \eqref{e3*.r2} hold for all $k\in\mathbb N$ and $j<k.$ Then from \eqref{e4*.r2} we obtain
$$f(x_n)=\sum_{k=1}^{\infty}a_{k}p(x_n)=\sum_{k=1}^{n}a_{k}p_{k}(x_n)=S_{n, \tilde p}(f)(x_n)$$
for every $n\in\mathbb N.$ Thus, $\tilde p=(p_k)_{k\in\mathbb N}$ is interpolating with the nodes $(x_k)_{k\in\mathbb N}.$
\end{proof}

\begin{corollary}
Let $X$ be an infinite compact subset of $\mathbb R$ and let $\tilde p=(p_k)_{k\in\mathbb N}$ be an interpolating Faber basis in $C_X.$ Then there is a unique sequence $(x_k)_{k\in\mathbb N}$ of distinct points of $X$ such that $\tilde p$ is interpolating with nodes $(x_k)_{k\in\mathbb N}.$
\end{corollary}
\begin{proof}
Let $\tilde p$ be interpolating with nodes $(x_k)_{k\in\mathbb N}$. By Lemma~\ref{lem1.r2} the point $x_1$ is the unique zero of the polynomial $p_2,$ the point $x_2$ can be characterized as the unique point of $X$ for which $p_{3}(x_2)=0$ and $p_{2}(x_2)\ne 0$ an so on.
\end{proof}

Lemma~\ref{lem1.r2} implies also the following
\begin{proposition}\label{prop1.r2}
Let $X$ be an infinite compact subset of $\mathbb R$. If $\tilde p=(p_k)_{k\in\mathbb N}$ be an interpolating Faber basis in $C_X$ with nodes $(x_k)_{k\in\mathbb N},$ then for every sequence $\tilde\lambda=(\lambda_k)_{k\in\mathbb N}$ of nonzero real numbers the sequence
$$\tilde\lambda\tilde p=(\lambda_{k}p_k)_{k\in\mathbb N}$$ is also an interpolating Faber basis with the same nodes $(x_k)_{k\in\mathbb N}$. Conversely, if $\tilde q=(q_k)_{k\in\mathbb N}$ and $\tilde p=(p_k)_{k\in\mathbb N}$ are interpolating Faber bases with the same nodes, then there is a unique sequence $\tilde \mu =(\mu_k)_{k\in\mathbb N}$ of nonzero real numbers such that
$$\tilde q=\tilde\mu\tilde p=(\mu_{k}p_{k})_{k\in\mathbb N}.$$
\end{proposition}

For given nodes $(x_k)_{k\in\mathbb N},$ the interpolating Faber basis $\tilde p=(p_k)_{k\in\mathbb N}$, if such a basis exists, can be uniquely determined by the natural normalization $$p_{k}(x_k)=1$$ for every $k\in\mathbb N.$

\begin{definition}\emph{\cite{Ob}}
A Faber basis $\tilde p=(p_k)_{k\in\mathbb N}$ is called a Lagrange basis with respect to the
sequence $(x_k)_{k\in\mathbb N}$ if
\begin{equation}\label{e1*.r2}
p_{k}(x_k)=1 \quad \mbox{and} \quad p_k(x_j)=0
\end{equation}
for all $k\in\mathbb N$ and $j< k.$
\end{definition}

The following example gives us another condition of uniqueness of interpolating Faber basis corresponding to given nodes. Recall that a polynomial is \emph{monic} if its leading coefficient is equal to $1.$

\begin{example}
Let $\tilde\pi=(\pi_k)_{k\in\mathbb N}$ be an interpolating Faber basis with nodes $(x_k)_{k\in\mathbb N}$ and monic polynomials $\pi_k.$ Then $\pi_1, \pi_2, ..., \pi_k, ...$ are the Newton polynomials,
\begin{equation}\label{Newpol.r2}
\pi_{k}(x)=
\begin{cases}
         1 & \mbox{if} $ $ k=1\\
         \prod_{j=1}^{k-1}(x- x_j)& \mbox{if}$ $ k\ge 2. \\
         \end{cases}
\end{equation}
The sequence $\tilde p=(p_k)_{k\in\mathbb N}$,
\begin{equation}\label{e12*.r2}
p_{k}=
\begin{cases}
         1 & \mbox{if} $ $ k=1\\
         \frac{\pi_k}{\prod_{j=1}^{k-1}(x_k- x_j)}& \mbox{if}$ $ k\ge 2 \\
         \end{cases}
\end{equation}
is a Lagrange basis with respect to $(x_k)_{k\in\mathbb N}.$
\end{example}

\begin{theorem}\label{vspth.r2}
Let $X$ be an infinite compact subset of $\mathbb R$ and let $(x_k)_{k\in\mathbb N}$ be a sequence of distinct points of $X.$ The following two statements are equivalent.
\newline \emph{(i)} There is an interpolating Faber basis with the nodes $(x_k)_{k\in\mathbb N}.$
\newline \emph{(ii)} For every $f\in C_X$ we have
\begin{equation}\label{eqv.r2}
f=\sum_{k=1}^{\infty}f[x_1, ..., x_k]\pi_k
\end{equation}
where, for each $k\in\mathbb N,$ $\pi_{k}$ is the Newton polynomials defined by \eqref{Newpol.r2} and $f[x_1, ..., x_k]$
is the divided difference of the function $f,$
$$f[x_1]=f(x_1),\,\,f[x_1, x_2]=\frac{f(x_1)}{x_{1}-x_{2}}+\frac{f(x_2)}{x_{2}-x_{1}}, ..., $$ $$f[x_1, ..., x_k]=\sum_{j=1}^{k}\frac{f(x_j)}{\prod_{i=1, i\ne j}^{k}(x_j - x_i)}.$$
\end{theorem}
\begin{proof}
(i)$\Rightarrow$(ii). If (i) holds, then by Lemma~\ref{lem1.r2} $\tilde\pi=(\pi_k)_{k\in\mathbb N}$ is an interpolating Faber basis in $C_X$ with nodes $(x_k)_{k\in\mathbb N}.$ Consequently, for every $f\in C_X$ there is a unique sequence $(y_k)_{k\in\mathbb N}$ such that
\begin{equation}\label{eqvsp.r2}
f=\sum_{k=1}^{\infty}y_{k}\pi_{k}.
\end{equation}
Since the basis $\tilde\pi$ is interpolating, we have
\begin{equation}\label{system}
\begin{cases}
         y_{1}\pi_{1}(x_1)=f(x_1),\\
         y_{1}\pi_{1}(x_2)+y_{2}\pi_{2}(x_2)=f(x_2), \\
         ................................................................\\
         y_{1}\pi_{1}(x_k)+y_{2}\pi_{2}(x_k)+ ... + y_{k}\pi_{k}(x_k)=f(x_k).
         \end{cases}
\end{equation}
The polynomial
\begin{equation*}\label{eqvvsp1.r2}
f[x_1]\pi_1+ ... + f[x_1, ..., x_k]\pi_k
\end{equation*}
coinsides with the function $f$ at the points $x_1, ..., x_k.$ (See Theorem~1.1.1 and formula (1.19) in \cite{JP} for details). Since linear system \eqref{system} has a unique solution, we have
\begin{equation}\label{eqvvsp1.r2}
y_1 = f[x_1], ..., y_{k}=f[x_1, ..., x_k].
\end{equation}
Equality \eqref{eqv.r2} follows.

(ii)$\Rightarrow$(i). Let (ii) hold. Then, the sequence $\tilde\pi=(\pi_k)_{k\in\mathbb N}$ is an interpolating Faber basis in $C_X$ if and only if \eqref{eqvsp.r2} implies \eqref{eqvvsp1.r2} for every $f\in C_X$ and every $k\in\mathbb N,$ that follows from the uniqueness of solutions of \eqref{system}.
\end{proof}
\begin{theorem}\label{r2Th2.3}
Let $X$ be an infinite compact subset of $\mathbb R$ and let $\mathfrak M=\{x_{k, n}\}$ be an interpolation matrix with the nodes in $X.$ The following conditions are equivalent.
\newline \emph{(i)} The space $C_X$ admits a Faber basis $\tilde p=(p_k)_{k\in\mathbb N}$ such that the equality
\begin{equation}\label{e2r2}
S_{n, \tilde p}=\mathfrak L_{n, \mathfrak M}
\end{equation}
holds for every $n\in\mathbb N.$
\newline \emph{(ii)} The sequence $(\Lambda_{n, X}(\mathfrak M))_{n\in\mathbb N}$ is bounded and there is a sequence $(x_k)_{k\in\mathbb N}$ of distinct points of $X$ such that for any $n\ge 2$ the tuple $(x_{1, n}, ..., x_{n, n})$ is a permutation of the set $\{x_1, ..., x_n\}.$
\end{theorem}
\begin{proof}
(i)$\Rightarrow$(ii). Let $\tilde p=(p_k)_{k\in\mathbb N}$ be a Faber basis in $C_X$ and let \eqref{e2r2} hold for every $n\in\mathbb N.$ The partial sum operators are bounded for every Faber basis. (See, for example, \cite[Proposition~1.1.4]{SZ}). Hence, we have $$\sup_{n}\parallel S_{n, \tilde p}\parallel<\infty.$$
The last inequality and \eqref{e2r2} imply the boundedness of the sequence $(\Lambda_{n, X}(\mathfrak M))_{n\in\mathbb N}.$ Now to prove (ii) it suffices to show that for every $n\ge 2$ and every $k_1\le n$ there is $k_2\le n+1$ such that $$x_{k_1, n}=x_{k_2, n+1}$$ holds. Suppose that, on the contrary, there is $n\ge 2$ and $k_1\in\{1, .., n\}$ such that $$x_{k_1, n}\ne x_{k_2, n+1}$$ for all integer numbers $k_2\in\{1, ..., n+1\}.$ We can find a function $f\in C_X$ satisfying the equalities $$f(x_{k_1, n})=1\quad \mbox{and} \quad f(x_{1, n+1})=f(x_{2, n+1})=...=f(x_{n+1, n+1})=0.$$ These equalities imply that $$\mathfrak L_{n+1, \mathfrak M}(f)=L_{n}(f, \mathfrak M, \cdot)=0 \quad\mbox{and}\quad \mathfrak L_{n, \mathfrak M}(f)=L_{n-1}(f, \mathfrak M, \cdot)\ne 0.$$ Now, using the obvious equality $$S_{n, \tilde p}\circ S_{n+1, \tilde p}=S_{n, \tilde p}$$ and \eqref{e2r2} we obtain the contradiction
\begin{equation*}
0\ne \mathfrak L_{n, \mathfrak M}(f)=S_{n, \tilde p}(f)=S_{n, \tilde p}(S_{n+1, \tilde p}(f))=S_{n, \tilde p}(\mathfrak L_{n+1, \mathfrak M}(f))=S_{n, \tilde p}(0)=0.
\end{equation*}
Statement (ii) follows.

(ii)$\Rightarrow$(i). Let (ii) hold. The boundedness of the sequence $(\Lambda_{n, X}(\mathfrak M))_{n\in\mathbb N}$ implies that
\begin{equation}\label{e3r2}
\lim_{n\to\infty}\parallel f- L_{n}(f, \mathfrak M, \cdot)\parallel_{X}=0
\end{equation}
holds for every $f\in C_X.$ (See Proposition~\ref{Prop1.1}). Since the Lagrange interpolation polynomial $L_{n}(f, \mathfrak M, \cdot)$ is invariant with respect to arbitrary permutation of the nodes $x_{1, n+1}, ..., x_{n+1, n+1},$ we may suppose that
$$x_{1, n+1}=x_1, \ , x_{2, n+1}=x_2, \, ..., \, x_{n+1, n+1}=x_{n+1}$$ for every $n\in\mathbb N.$ Using the Newton polynomials $\pi_k$ (see \eqref{Newpol.r2})
we may write the polynomial $L_{n}(f, \mathfrak M, \cdot)$ in the form
\begin{equation}\label{e5r2}
L_{n}(f, \mathfrak M, \cdot)=f[x_1]\pi_1+f[x_1, x_2]\pi_2+...+f[x_1, x_2, ..., x_{n+1}]\pi_{n+1}.
\end{equation}
Hence, we have the representation $f=\sum_{k=1}^{\infty}f[x_1, ..., x_k]\pi_k.$ Now, (i) follows from Theorem~\ref{vspth.r2}.
\end{proof}

\begin{corollary}\label{r2c2.4}
Let $X$ be an infinite compact subset of $\mathbb R$ and let $\mathfrak M\subseteq X$ be an interpolation matrix with bounded $(\Lambda_{n, X}(\mathfrak M))_{n\in\mathbb N}.$ Then the following conditions are equivalent.
\newline \emph{(i)} There is a Faber basis of $C_X$ such that \eqref{e2r2} holds for every $n\in\mathbb N.$
\newline \emph{(ii)} The equality $$ \mathfrak L_{n, \mathfrak M}\circ  \mathfrak L_{n+1, \mathfrak M}=\mathfrak L_{n+1, \mathfrak M}\circ \mathfrak L_{n, \mathfrak M}$$ holds for every $n\in\mathbb N.$
\newline \emph{(iii)} The inequality $$\emph{deg} L_{n}(f, \mathfrak M, \cdot)\ge \emph{deg} L_{n-1}(f, \mathfrak M, \cdot)$$ holds for every $n\in\mathbb N$ and every $f\in C_{X}.$
\end{corollary}
\begin{proof}
The implications (i)$\Rightarrow$(ii) and (i)$\Rightarrow$(iii) follow directly from Definition~\ref{r2d2.1}. The proofs of (ii)$\Rightarrow$(i) and (iii)$\Rightarrow$(i) are similar to the proof (i)$\Rightarrow$(ii) in Theorem~\ref{r2Th2.3}
\end{proof}
\begin{remark}
Statements (ii) and (iii) of Corollary~\ref{r2c2.4} can be considered as some special cases of Lemma~4.7 in \cite{FHHMZ} and Theorem~20.1 in \cite{SiI} respectively.
\end{remark}

\begin{lemma}\label{vspLem}
Let $X$ be an infinite compact subset of $\mathbb R.$ The following statements are equivalent for arbitrary Faber bases $\tilde p=(p_k)_{k\in\mathbb N}$ and $\tilde q=(q_k)_{k\in\mathbb N}$ in $C_X.$
\newline \emph{(i)}There is a sequence $\tilde\lambda=(\lambda_k)_{k\in\mathbb N}$ of nonzero numbers such that
\begin{equation}\label{leq1r2}
p_k=\lambda_{k}q_k
\end{equation}
holds for every $k\in\mathbb N.$
\newline \emph{(ii)} The equality
\begin{equation}\label{leq2r2}
S_{n, \tilde p}=S_{n, \tilde q}
\end{equation}
holds for every $n\in\mathbb N.$
\end{lemma}
\begin{proof}
The implication (i)$\Rightarrow$(ii) follows from the definition of the Faber bases in $C_X.$

Let (ii) hold. Equality \eqref{leq1r2} is trivial if $k=1.$ If $k\ge 2,$ then there are $\lambda_1, ..., \lambda_k \in\mathbb R$ such that
$$p_{k}=\sum_{i=1}^{k}\lambda_{i}q_{i}=\lambda_{k}q_{k}+S_{k-1, \tilde q}(p_k).$$
Using \eqref{leq2r2} we obtain
$$S_{k-1, \tilde q}(p_k)=S_{k-1, \tilde p}(p_k)=0.$$
Hence $p_k=\lambda_{k}q_{k}$ holds. Moreover, we have $\lambda_k\ne 0$ because $$\textrm{deg} p_k= \textrm{deg} q_k=k-1.$$
\end{proof}
The following theorem is a dual form of Theorem~\ref{r2Th2.3} and it can be considered as the main result of the third section of the paper.
\begin{theorem}\label{r2th2.6}
Let $X$ be an infinite compact subset of $\mathbb R$ and let $\tilde p=(p_k)_{k\in\mathbb N}$ be a Faber basis in $C_X.$ The following conditions are equivalent.
\newline \emph{(i)} There exists an interpolation matrix $\mathfrak M\subseteq X$ such that  equality \eqref{e2r2} holds for every $n\in\mathbb N.$
\newline \emph{(ii)} The basis $\tilde p$ is interpolating.
\end{theorem}

\begin{proof}
(i)$\Rightarrow$(ii). Let $\mathfrak M=\{x_{n, k}\}$ be an interpolation matrix such that $\mathfrak M\subseteq X$ and the equality
\begin{equation}\label{e1th2.14}
\mathfrak L_{n, \mathfrak M}=S_{n, \tilde p}
\end{equation}
holds for every $n\in X.$ Using Theorem~\ref{r2Th2.3} we can suppose that there is a sequence $(x_k)_{k\in\mathbb N}$ of distinct points of $X$ such that
$$x_{k, n}=x_k$$
for all $n\ge 1$ and $k\in\{1, ..., n\}.$ To prove (ii) it suffices to show that $\tilde p$ is interpolating with nodes $(x_k)_{k\in\mathbb N}.$ As in the proof of implication (ii)$\Rightarrow$(i) from Theorem~\ref{r2Th2.3} we obtain that the basis $\tilde\pi=(\pi_k)_{k\in\mathbb N}$ consisting of the corresponding Newton polynomials is an interpolating Faber basis with the nodes $(x_k)_{k\in\mathbb N}$ for which the equality
\begin{equation}\label{e2Th2.14}
\mathfrak L_{n, \mathfrak M}=S_{n, \tilde \pi}
\end{equation}
holds for every $n\in\mathbb N.$ (See equality \eqref{e5r2}). By Lemma~\ref{vspLem}, it follows from \eqref{e1th2.14} and \eqref{e2Th2.14} that there is a sequence $(\lambda_k)_{k\in\mathbb N}$ of nonzero real numbers such that
$$p_{k}=\lambda_{k}\pi_k$$
holds for every $k\in\mathbb N.$ Since $\tilde\pi$ is an interpolating Faber basis with nodes $(x_k)_{k\in\mathbb N},$ Proposition~\ref{prop1.r2} implies that $\tilde p$ is also interpolating with the same nodes.

(ii)$\Rightarrow$(i). Suppose that $\tilde p=(p_k)_{k\in\mathbb N}$ is interpolating with nodes $(x_k)_{k\in\mathbb N}.$ If $$\tilde p=\tilde\pi,$$ where $\tilde\pi=(\pi_{k})_{k\in\mathbb N}$ is the interpolating basis consisting of the Newton polynomials, then using Theorem~\ref{vspth.r2} we can show that \eqref{e2r2} holds for all $n\in\mathbb N$ with
$$\mathfrak M=\{x_{k, n}\}, \quad x_{k, n}=x_{k}, \quad n\in\mathbb N, \, k\in\{1, ..., n\}.$$
The case of an arbitrary interpolating Faber basis $\tilde p=(p_k)_{k\in\mathbb N}$ can be reduced to the case $\tilde p=\tilde\pi$ with the help of Lemma~\ref{vspLem} and Proposition~\ref{prop1.r2}.
\end{proof}

\medskip

\noindent \textbf{Acknowledgements.} The authors were supported by Grant FP7-People-2011-IRSES Project 295164, EUMLS:
EU-Ukrainian Mathematicians for Life Sciences. The first two authors were also partially supported by the State Fund For Fundamental Research (Ukraine),
Project 20570 and by Project 15-1bb$\setminus$19 ``Metric Spaces, Harmonic Analysis of Functions and Operators and Singular and Nonclassic Problems  of Differential Equations'' (Donetsk National University, Vinnitsia, Ukraine).

\bigskip

\small

\textbf{Viktoriia Bilet}

Institute of Applied Mathematics and Mechanics, NAS of Ukraine, Dobrovolskogo Str. 1, 84100, Sloviansk, Ukraine,

E-mail addresses: biletvictoriya@mail.ru, viktoriiabilet@gmail.com;

\bigskip

\textbf{Oleksiy Dovgoshey}

Institute of Applied Mathematics and Mechanics, NAS of Ukraine, Dobrovolskogo Str. 1, 84100, Sloviansk, Ukraine,

E-mail addresses: aleksdov@mail.ru, aleksdov@gmail.com;

\bigskip

\textbf{J\"{u}rgen Prestin}

Universit\"{a}t zu L\"{u}beck, Institut f\"{u}r Mathematik, Ratzeburger Allee 160, 23562, L\"{u}beck, Germany,

E-mail address: prestin@math.uni-luebeck.de.
\end{document}